\documentclass[reqno]{amsart}
\usepackage[utf8]{inputenc}

\usepackage{amssymb}
\usepackage{graphicx}
\usepackage{latexsym}
\usepackage{stmaryrd}
\usepackage{enumitem}
\usepackage[bookmarks]{hyperref}
\usepackage{multirow}
\usepackage{xcolor}
\usepackage{relsize}
\usepackage{comment}
\usepackage{xifthen}
\usepackage{mathrsfs}
\usepackage{svg}
\usepackage{mathtools}
\usepackage{tikz-cd}
\colorlet{darkishRed}{red!80!black}
\colorlet{darkishBlue}{blue!60!black}
\colorlet{darkishGreen}{green!60!black}
\renewcommand{\leq}{\leqslant}
\renewcommand{\geq}{\geqslant}
\renewcommand{\ge}{\geq}
\renewcommand{\le}{\leq}

\usepackage{hyperref}
\hypersetup{
    unicode,
    colorlinks,
    linkcolor={red!60!black},
    citecolor={green!60!black},
    urlcolor={blue!60!black}
}

\renewcommand{\subset}{\subseteq}
\renewcommand{\supset}{\supseteq}

\newcommand{\Abs}[1]{\partial {#1}}

\newcommand{ \N } { \mathbb{N} }

\newcommand{\eps}{\varepsilon}

\makeatletter

\def\calCommandfactory#1{%
   \expandafter\def\csname c#1\endcsname{\mathcal{#1}}}
\def\frakCommandfactory#1{%
   \expandafter\def\csname frak#1\endcsname{\mathfrak{#1}}}

\newcounter{ctr}
\loop
  \stepcounter{ctr}
  \edef\X{\@Alph\c@ctr}
  \expandafter\calCommandfactory\X
  \expandafter\frakCommandfactory\X
  \edef\Y{\@alph\c@ctr}
  \expandafter\frakCommandfactory\Y
\ifnum\thectr<26
\repeat

\renewcommand{\cC}{\mathscr{C}}

\newcommand{\script}{\mathcal}
\newcommand{\parentheses}[1]{{\left( {#1} \right)}}

\newcommand{\p}{\parentheses}

\newcommand{\Set}[1]{{\left\lbrace {#1} \right\rbrace}}

\def\set#1:#2{\Set{{#1} \colon {#2}}}

\newtheorem{theorem}{Theorem}[section]

\newtheorem{lemma}[theorem]{Lemma}

\theoremstyle{definition}

\theoremstyle{remark}

\setenumerate{label={\normalfont (\roman*)},itemsep=0pt}

\usepackage{geometry}
\begin{document}
\title[Tree-decompositions displaying topological ends]{Constructing tree-decompositions that display \\ all topological ends}

\author{Max Pitz}
\address{Universit\"at Hamburg, Department of Mathematics, Bundesstra\ss e 55 (Geomatikum), 20146 Hamburg, Germany}
\email{max.pitz@uni-hamburg.de}



\begin{abstract}
We give a short, topological proof that all graphs admit tree-decompositions displaying their topological ends. 
\end{abstract}

\vspace*{-1cm}
\maketitle

\section{Introduction}

Historically, one of the strongest driving force behind investigating tree structure of  infinite graphs has been Halin's \emph{end-faithful spanning tree conjecture} from the 1960's \cite{Halin_Enden64}, until is was refuted in the early 1990's independently by Seymour \& Thomas \cite{seymour1991end} and by Thomassen \cite{thomassen1992infinite}. 

However, in a 50-page  breakthrough from 2014 (published in 2019), Carmesin \cite{carmesin2019all} proved a
general theorem how to separate ends in graphs that generalised corresponding results from Cayley graphs of finitely generated groups, and applied it to establish that all graphs admit tree-decompositions displaying their topological ends. The latter result also implies that every connected graph has a spanning tree that is end-faithful for all topological ends. Together, these results settled a problem of Diestel from 1992 \cite{diestel1992end} and gave the first satisfying answer to the above conjecture of Halin's in amended form. 

The purpose of this note is to offer a short proof for both of Carmesin's applications that is motivated by topological instead of algebraic considerations. It is based on a technique  that we call \emph{enveloping} a given set of vertices, which enables one to expand any set of vertices without changing the number of ends in its closure, and which we expect to have further applications in the study of tree-decompositions and end structure of infinite graphs.

\section{Background on ends and topological ends}
\label{sec_2}

\subsection{Ends}
For graph theoretic terms we follow \cite{Bible}, and in particular \cite[Chapter~8]{Bible} for ends of graphs. A~$1$-way infinite path is called a \emph{ray} and the subrays of a ray are its \emph{tails}. Two rays in a graph $G = (V,E)$ are \emph{equivalent} if no finite set of vertices separates them; the corresponding equivalence classes of rays are the \emph{ends} of $G$. 
If $X \subseteq V$ is finite and $\eps$ is an end, 
there is a unique component of $G-X$ that contains a tail of every ray in $\eps$, which we denote by $C(X,\eps)$. 
Then $\eps$ \emph{lives} in the component $C(X,\eps)$. 

An end $\eps$ of $G$ is contained \emph{in the closure} of $M$, where $M$ is either  a subgraph of $G$ or a set of vertices of $G$, if for every finite vertex set $X\subset V$ the component $C(X,\eps)$ meets $M$.
We write $\Abs{M}$ 
for the set of ends of $G$ lying in the closure of $M$.

\subsection{Star-comb-lemma}
A \emph{comb} is the union of a ray $R$ (the comb's \emph{spine}) 
with infinitely many disjoint finite paths, possibly trivial, that have precisely their first vertex on~$R$. 
The last vertices of those paths are the \emph{teeth} of this comb.
Given a vertex set $U$, a \emph{comb attached to} $U$ is a comb with all its teeth in $U$, and a \emph{star attached to} $U$ is a subdivided infinite star with all its leaves in $U$.

\begin{lemma}[Star-comb lemma]\label{StarCombLemma}
Let $U$ be an infinite set of vertices in a connected graph $G$.
Then $G$ contains either a comb 
attached to~$U$ or a star attached to~$U$.
\end{lemma}

\subsection{Dominated ends}
A vertex $v$ \emph{dominates} an end $\eps$ if there is a star with center $v$ attached to some (equivalently: any) ray of $\eps$. In this case, we say that $\eps$ is \emph{dominated}. Diestel and K\"uhn proved that the undominated ends of a graph $G$ correspond precisely to the topological ends in the sense of Freudenthal~\cite{diestel2003graph}, and hence the undominated ends are also called \emph{topological} ends of a graph.

We will need the following standard result isolated from the proof of \cite[Theorem~2.2]{diestel2003graph}:
\begin{lemma}
\label{lem:undominatedends}
Suppose $\set{X_n}:{n \in \N}$ is a sequence of disjoint finite sets of vertices in a graph $G$ and suppose $C_n$ are connected components of $G-X_n$ such that $C_{n} \supseteq 
C_{n+1} \cup X_{n+1} $ for all $n \in \N$. Then there is a unique end of $G$ that lives in all $C_n$, and this end is undominated.
\end{lemma}

We will also need the following routine result relating the closure operator to dominated ends.

\begin{lemma}
\label{lem:closuredom}
Let $W$ be a finite set of vertices in a graph $G$, and $U= N(W)$ its neighbourhood. Then $\Abs{U}$ consists of precisely those ends of $G$ that are dominated by a vertex in $W$.
\end{lemma}

\begin{proof}
This follows from the well-known fact that $\eps \in \Abs{U}$ if and only if there is  a comb attached to $U$ with spine in $\eps$.
\end{proof}

\subsection{End-faithful spanning trees and tree-decompositions}

A rooted spanning tree $T$ of a graph $G$ is \emph{end-faithful} for a set $\Psi$ of ends of $G$ if for each end $\eps \in \Psi$ there is a unique rooted ray $R$ in $T$ with $R \in \eps$.

A \emph{tree-decomposition} of a graph $G$ is a pair $\script{T} = \p{T,\script{V}}$ where $T$ is a tree and $\cV = (V_t \colon t \in T)$ is a family of vertex sets of $G$ called \emph{parts} such that the following hold (see also \cite[\S12.3]{Bible}):
\begin{enumerate}[label=(T\arabic*)]
\item for every vertex $v$ of $G$ there exists $t \in T$ such that $v \in V_t$;
\item for every edge $e$ of $G$ there exists $t \in T$ such that $e \in G[V_t]$; and
\item $V_{t_1} \cap V_{t_3} \subseteq V_{t_2}$ whenever $t_2$ lies on the $t_1-t_3$ path in $T$.
\end{enumerate}
Let $e = xy$ be any edge of $T$ and let $T_x$ and $T_y$ be the two components of $T-e$ with $x \in T_x$ and $y \in T_y$. Each edge $e=xy$ of $T$ in a tree-decomposition gives rise to a separator $X_e:=V_x \cap V_y$ called the separator \emph{induced by} the edge $e$, which separates $A_x=\bigcup_{t \in T_x} V_t$ from $A_y=\bigcup_{t \in T_y} V_t$. 
The tree-decomposition has \emph{finite adhesion} if all separators of $G$ induced by the edges of $T$ are finite. 

Given a tree-decomposition $\script{T} = \p{T,\script{V}}$ of finite adhesion of $G$, any end $\eps$ of $G$ orients each edge $e=xy$ of $T$ according to whether $\eps$ lives in a component of $G[A_x] - X_e$ or $G[A_y]-X_e$. This orientation of $T$ points towards a node of $T$ or to an end of $T$, and $\eps$ \emph{lives} in that part for that node or that end, respectively. 
%
Then $\script{T}$ \emph{displays} a set $\Psi$ of ends of $G$ if in every end of $T$ there lives a unique end and it is in $\Psi$, and conversely every end of $\Psi$ lives in some end of $T$.

Finally, a \emph{rooted tree-decomposition} is $\script{T} = \p{T,\script{V}}$ where the decomposition tree $T$ is rooted. A rooted tree-decomposition is said to have \emph{upwards disjoint separators} if the induced separators $X_e$ and $X_{e'}$ for any two distinct edges $e < e'$ comparable in the tree order of $T$ are disjoint.

\subsection{Rooted trees containing a set of vertices cofinally}

Recall that a subset $X$ of a poset $P=(P,{\le})$ is \emph{cofinal} in $P$, and ${\le}$, if for every $p\in P$ there is an $x\in X$ with $x\ge p$.
We say that a rooted tree $T\subset G$ contains a set $U$ \emph{cofinally} if $U\subset V(T)$ and $U$ is cofinal in the tree order of~$T$. The main assertion of the following is an immediate corollary of~\cite[Lemma~2.13]{StarComb1StarsAndCombs}.

\begin{lemma}\label{lem:cofinalConcentrated}
Let $G$ be any graph, and let $U\subset V(G)$ be a set of vertices. 
If $T\subset G$ is a rooted tree that contains $U$ cofinally, then $\Abs{T} = \Abs{U}$. 
\end{lemma}

\section{Envelopes for sets of vertices}
\label{sec_envelope}

Let $G$ be a connected graph. Given a subgraph $C \subseteq G$, write $N(C)$ for the set of vertices in $G -C$ with a neighbour in $C$. An \emph{adhesion set} of a set of vertices or a subgraph $U \subset G$ is any subset of the form $N(C)$ for a component $C$ of $G-U$. The set or subgraph $U$ is said to have \emph{finite adhesion} in $G$ if all its adhesion sets are finite.

An \emph{envelope} for a set of vertices $U \subset V(G)$ is a set of vertices $U^* \supseteq U$ of finite adhesion such that $\Abs{ U^*} = \Abs{U}$. The following theorem has been developed by Kurkofka and the author in \cite[Theorem~3.1]{pitz2021Rep}. We take the opportunity here to present a somewhat different proof.

\begin{theorem}
\label{thm:envelope}
Any set of vertices in a connected graph admits a connected envelope. 
\end{theorem}

\begin{proof}
We use the following concepts. 
Let $W$ be any set of vertices. 
An \emph{external comb attached to} $W$ is the union of a ray~$R$ that avoids~$W$ together with infinitely many disjoint 
$R$--$W$ paths.
The last vertices of those paths in $W$ form the \emph{attachment set} of this external comb.
An \emph{external star attached to} $W$ is a subdivided infinite star with precisely its leaves in~$W$.
Its set of leaves is its \emph{attachment set}.
The \emph{interior} of an external star or comb attached to~$W$ is obtained from it by deleting~$W$.
We call a collection of external stars and combs attached to~$W$ \emph{internally disjoint} if all its elements have pairwise disjoint interior.

We recursively construct a sequence $(\,U_i \colon i < \omega_1\,)$ of sets of vertices in $G$ as follows.
Let $T\subset G$ be a rooted tree that includes $U$ cofinally, and put $U_0:=V(T)$.
If $U_i$ is already defined, we use Zorn's lemma to choose a maximal collection $\cC_i$ consisting of internally disjoint external stars and combs in~$G$ attached to~$U_i$, and let $U_{i+1} := U_i \cup V[\,\bigcup \cC_i\,]$. 
For limits $\ell < \omega_1$ we define $U_\ell := \bigcup_{i < \ell} U_i$.
We claim that $U^* := \bigcup_{i < \omega_1} U_i$ is a connected envelope for $U$.

First, $U_0$ is connected, and it follows by induction on~$i$ that every $U_i$ is connected, too. Hence, so is $U^*$.
Similarly, $\Abs{U_0} = \Abs{U}$ by choice of $T$ and Lemma~\ref{lem:cofinalConcentrated}, and it follows once again by induction on~$i$ that $\Abs{U_i} = \Abs{U}$ for every $i < \omega_1$. Indeed, consider an end $\eps \notin \Abs{U_0}$. Then there is a finite set of vertices $X$ such that $C(X,\eps)$ avoids $U_0$. But then  throughout the whole process, we will attach at most $|X|$ external combs or stars that intersect $C(X,\eps)$, as every one also has to intersect $X$ internally. Since every such star or comb will intersect $C(X,\eps)$ finitely, also $U^*$ intersects $C(X,\eps)$ finitely, witnessing $\eps \notin \Abs{U^*}$. This gives $\Abs{U^*} = \Abs{U}$ as desired.

To see that $U^*$ has finite adhesion, suppose for a contradiction that there is a component $C$ of $G-U^*$ with infinite neighbourhood.
Then by a routine application of the star-comb lemma (Lemma~\ref{StarCombLemma}), we find either an external star or comb attached to~$U^*$ whose interior is completely contained in~$C$. 
Its countable attachment set, however, already belongs to some $U_i$ with~$i < \omega_1$ (for $\omega_1$ has uncountable cofinality). But then the existence of this external star or comb contradicts the maximality of~$\cC_i$.
\end{proof}

\section{Tree-decompositions that displays all topological ends}

\begin{lemma}
\label{lem:refiningsequence}
Every connected graph $G$ admits a sequence of induced connected subgraphs $H_0 \subset H_1 \subset \cdots$ in $G$ all of finite adhesion such that
\begin{enumerate}
	\item $N(H_n) \subset H_{n+1}$ for all $n \in \N$ (implying $G = \bigcup_{n \in \N} H_n$),
	\item for all $n \in \N$, every topological end lives in a unique component of $G-H_n$, and
	\item for all $n \in \N$ and every component $C$ of $G-H_n$, the set $C \cap H_{n+1}$ 
	is connected.
\end{enumerate}
\end{lemma}

\begin{proof}
Let $H_0$ consist of some arbitrarily chosen singleton. If $H_n$ is already defined, consider some component $C$ of $G-H_n$. Let $W=N(C)$, and let $U_C$ be the neighbourhood of $W$ in $C$. Use Theorem~\ref{thm:envelope} inside $C$ to find a connected envelope $U^*_C$ of $U_C$. Define $H_{n+1}$ to consist of $H_n$ together with all connected envelopes $U^*_C$ for all components $C$ of $G-H_n$. Then $H_{n+1}$ has finite adhesion. 

Now properties (i)  and (iii) hold by construction. For (ii), consider a topological end $\eps$ of $G$. By induction assumption, $\eps$ lives in a unique component $C$ of $G-H_n$. By Lemma~\ref{lem:closuredom}, $\eps$ does not belong to $\Abs{U_C}$, and hence also not to $\Abs{U^*_C}$. By finite adhesion of $U^*_C$, there is a unique component $C'$ of $C - U^*_C$ in which $\eps$ lives. By definition of $H_{n+1}$, this component $C'$ is also a component of $G-H_{n+1}$ as desired.
\end{proof}

\begin{theorem}
\label{thm:td}
Every connected graph has a rooted tree-decomposition, of finite adhesion and into connected parts with upwards disjoint separators, that displays all topological ends.
\end{theorem}

\begin{proof}
The sequence $H_0 \subset H_1 \subset H_2 \subset \cdots$ from Lemma~\ref{lem:refiningsequence} gives rise to a tree decomposition $\p{T,\cV}$ of finite adhesion and into connected parts as follows: Write $\cC_n$ for the collection of components of $G-H_n$. The reverse inclusion relation `$\supset$' defines a tree order on the set $T:=\Set{G} \cup \bigcup_{n \in \N} \cC_n$ with root $G$; this will be our decomposition tree. The part corresponding to the root of $T$ will be $H_0$. The part corresponding to a node $C \in \cC_n$ of $T$ will be $N(C) \cup \p{C \cap H_{n+1}}$, which is connected by (i) and (iii). Then it is readily checked that all properties (T1) -- (T3) of a tree-decomposition are implied by (i), as is the property that this tree-decomposition has upwards disjoint separators.

To see that $(T,\cV)$ displays all topological ends, observe first that every rooted ray $t_0t_1t_2t_3,\ldots$  in $T$ gives rise to a nested sequence of non-empty components $C_{t_1} \supsetneq C_{t_2} \supsetneq \cdots$ with $C_{t_n} \in \cC_n$ for $n \in \N$, 
such that $C_{t_n} \supseteq 
C_{t_{n+1}} \cup N(C_{t_{n+1}})$ for all $n \in \N$ by property (i). Hence by Lemma~\ref{lem:undominatedends} there is a unique end that lives in all $C_{t_n}$, which is undominated. Conversely, every topological end of $G$ lives in a unique connected component $C_{t_n}$ of $G - H_n$ by (ii), and so $t_1 t_2 t_3\ldots$ is a ray in $T$ corresponding to this end.
\end{proof}

\begin{theorem}
Every connected graph admits a spanning tree that is end-faithful for the topological ends.
\end{theorem}

\begin{proof}
Given the sequence $H_0 \subset H_1 \subset H_2 \subset \cdots$ from Lemma~\ref{lem:refiningsequence} we construct rooted trees $T_0 \subset T_1 \subset \cdots$ such that $T_n$ is a spanning tree of $H_n$ as follows: Let $T_0$ be a rooted spanning tree of $H_0$. If $T_n$ is already defined, extend $T_n$ to a spanning tree of $H_{n+1}$ by choosing, for every component $C$ of $G-H_n$ a spanning tree of $C \cap H_{n+1}$ by (iii), and attaching it to $T_n$ via a single edge $e_C$. 

We show that $T = \bigcup_{n \in \N} T_n$ is a spanning tree of $G$ that is end-faithful for the topological ends of $G$. First, every topological end $\eps$ of $G$ lives, for all $n$, in a unique connected component $C_{n}$ of $G - H_n$ by (ii). Since $C_{n+1} \cup N(C_{n+1}) \subset C_n$ by (i), the edges $e_{C_n}$ lie on a rooted ray $R$ of $T$, which satisfies $R \in \eps$. To see that $R$ is the unique ray of $T$ which belongs to $\eps$, suppose for a contradiction there was another rooted ray $R' \subset T$ belonging to $\eps$. Then $R$ and $R'$ are eventually disjoint, so choose $n \in \N$ such that $R \cap R' \subset H_n$. Since $e_{C_n} \in E(R) \setminus E(R')$, it follows that $R$ eventually belongs to $C_n$ while $R' \cap C_n = \emptyset$, contradicting that $R$ and $R'$ are equivalent.
\end{proof}

\bibliographystyle{abbrv}
\bibliography{TD}

\begin{thebibliography}{1}

\bibitem{StarComb1StarsAndCombs}
C.~Bürger and J.~Kurkofka.
\newblock {Duality theorems for stars and combs I: Stars and combs}.
\newblock {\em Journal of Graph Theory}, 2021.

\bibitem{carmesin2019all}
J.~Carmesin.
\newblock All graphs have tree-decompositions displaying their topological
  ends.
\newblock {\em Combinatorica}, 39(3):545--596, 2019.

\bibitem{diestel1992end}
R.~Diestel.
\newblock The end structure of a graph: recent results and open problems.
\newblock {\em Disc.\ Math.}, 100(1):313--327, 1992.

\bibitem{Bible}
R.~Diestel.
\newblock {\em {Graph Theory}}.
\newblock Springer, 5th edition, 2015.

\bibitem{diestel2003graph}
R.~Diestel and D.~K{\"u}hn.
\newblock {Graph-theoretical versus topological ends of graphs}.
\newblock {\em J. Combin.\ Theory (Series B)}, 87(1):197--206, 2003.

\bibitem{Halin_Enden64}
R.~Halin.
\newblock {\"Uber unendliche Wege in Graphen}.
\newblock {\em Mathematische Annalen}, 157:125--137, 1964.

\bibitem{pitz2021Rep}
J.~Kurkofka and M.~Pitz.
\newblock A representation theorem for end spaces of infinite graphs.
\newblock https://arxiv.org/abs/2111.12670.

\bibitem{seymour1991end}
P.~Seymour and R.~Thomas.
\newblock An end-faithful spanning tree counterexample.
\newblock {\em Proceedings of the American Mathematical Society},
  113(4):1163--1171, 1991.

\bibitem{thomassen1992infinite}
C.~Thomassen.
\newblock Infinite connected graphs with no end-preserving spanning trees.
\newblock {\em Journal of Combinatorial Theory, Series B}, 54(2):322--324,
  1992.

\end{thebibliography}

\end{document}